\documentclass[11pt]{article}
\usepackage{amsmath,amssymb,amsthm,latexsym,color,epsfig}
\newtheorem{thm}{Theorem}

\newtheorem{lem}{Lemma}

\newcommand{\cA}{{\cal A}}
\newcommand{\BB}{{\cal B}}

\newcommand{\eps}{{\varepsilon}}
\newcommand{\whp}{{\bf whp}\ }

\DeclareMathOperator{\Bin}{Bin}
\DeclareMathOperator{\Prob}{Pr}
\DeclareMathOperator{\E}{\mathbb{E}}

\begin{document}
\date{}
\title{Component sizes in the supercritical percolation on the binary cube}
\author{Michael Krivelevich\thanks{
School of Mathematical Sciences,
Tel Aviv University, Tel Aviv 6997801, Israel.
Email: {\tt krivelev@tauex.tau.ac.il}.  }
}
\vspace{-2cm}
\maketitle
\begin{abstract}
We present a relatively short and self-contained proof of the classical result on component sizes in the supercritical percolation on the high dimensional binary cube, due to Ajtai, Koml\'os and Szemer\'edi (1982) and to Bollob\'as, Kohayakawa and \L uczak (1992).
\end{abstract}

The purpose of this expository note is to present a fairly short and essentially self-contained proof of the classical result of Ajtai, Koml\'os and Szemer\'edi \cite{AKS82} and of Bollob\'as, Kohayakawa and \L uczak \cite{BKL92} about typical component sizes in the supercritical percolation on the binary cube. The argument relies on several beautiful ideas presented in these two papers.

Recall that the $d$-dimensional binary cube $Q^d$ is defined as follows: its vertex set is $V=\{0,1\}^d$, and two vertices $\bar{x},\bar{y}\in V$ are connected by an edge in $Q^d$ if they differ in exactly one coordinate. Thus, $Q^d$ is a $d$-regular graph on $n:=2^d$ vertices. Throughout this note we assume that $d$ is an asymptotic parameter tending to infinity.

For $p=p(d)\in [0,1]$, form a random subgraph $Q^d_p$ of $Q^d$ by retaining every edge of $Q^d$ independently and with probability $p$.

We start with the much easier to handle subcritical case; peculiarly enough, this case will turn out to be useful for our argument in the supercritical regime. This case is immediately settled by the following general statement.

\begin{thm}\label{th0}
Let $\eps>0$ be a small enough constant. Let $G$ be a graph on $n$ vertices of maximum degree $d\ge 3$. Let $p=\frac{1-\eps}{d-1}$, and form a random subgraph $G_p$ of $G$ by retaining every edge of $G$ independently and with probability $p$. Then \whp all connected components of $G_p$ are of size at most $\frac{9\ln n}{\eps^2}$.
\end{thm}

\noindent{\bf Remark.} In the above statement we purportedly put the edge probability at $p=\frac{1-\eps}{d-1}$ instead of $p=\frac{1-\eps}{d}$, to make this statement more applicable to the case of constant degree $d$, where one expects the critical probability $p^*$ to be close to $1/(d-1)$. For constant $\eps$ and growing $d$, the difference between $1/(d-1)$ and $1/d$ is insignificant.

\begin{proof}
Let $m=|E(G)|\le \frac{dn}{2}<n^2$. We run the Breadth First Search (BFS) Algorithm on $G$, feeding it with the sequence of random bits $\bar{X}=(X_i)_{i=1}^m$, where $X_i$'s are independent Bernoulli$(p)$ random variables. Whenever the algorithm queries the existence of the $i$-th random edge (in its count of exposed edges) $e\in E(G)$, we reach to the bit $X_i$ and assume that $e\in E(G_p)$ if and only if $X_i=1$.

Suppose towards a contradiction that $G_p$ has a component $K$ with more than $k$ vertices, the value of $k$ will be chosen later. Consider the moment in the algorithm's execution where it is in the midst of revealing $K$ and has just discovered the $(k+1)$-th vertex of $K$. In the phase of discovering $K$ up until this moment, only the edges of $G$ incident to the set $K_0$ of the first $k$ discovered vertices of $K$ have been queried, and $k-1$ edges of $G_p$ induced by this set of vertices have been revealed, implying that the number of edges of $G$ spanned by $K_0$ satisfies: $e_G(K_0)\ge k-1$. Recalling our assumption on the maximum degree of $G$, we conclude that the number of edges of $G$ queried so far in this phase is at most
$$
k\cdot d - e_G(K_0)\le kd-(k-1)=k(d-1)+1\,.
$$
It follows that the sequence $\bar{X}$ contains $k(d-1)+1$ consecutive bits $X_i$, out of which at least $k$ are equal to 1. For a given interval of length $k(d-1)+1$, the probability of having at least $k$ ones is
$$
\Prob\left[\Bin\left(k(d-1)+1,\frac{1-\eps}{d-1}\right)\ge k\right]\le e^{-\frac{\eps^2k}{4}}
$$
by the standard Chernoff-type bounds. Union bounding, we have that the probability of having an interval of length $k(d-1)+1$ in $[m]$ with at least $k$ ones is at most $m\cdot  e^{-\frac{\eps^2k}{4}}$. Taking $k=\frac{9\ln n}{\eps^2}$, we see that \whp $\bar{X}$ has no such interval, meaning that \whp $G_p$ has no component larger than $k$.
\end{proof}

\medskip

We now proceed to the main part of this note, the supercritical regime.
For $c>1$, define $y=y(c)$ to be the unique solution in $(0, 1)$ of the equation
\begin{equation}\label{eq1}
y=1-\exp\{-cy\}\,.
\end{equation}
Note that for $c=1+\epsilon$ and $\eps>0$ small enough, we can estimate $y=(1+o_{\eps}(1))2\eps$.

The following is a fundamental statement on the typical component sizes in the random subcube $Q^d_p$ for the supercritical case $p=\frac{c}{d}$, $c>1$ a constant.

\begin{thm}\label{th1}\cite{AKS82,BKL92}
Let $c>1$ be a  constant, and let $p=\frac{c}{d}$. Form a random subgraph $Q^d_p$ of the $d$-dimensional binary cube $Q^d$ by retaining every edge of $Q^d$ independently and with probability $p$. Then, \whp the graph $Q^d_p$ has a connected component $L_1$ whose size is asymptotic to $yn$, where $y=y(c)$ is as defined in (\ref{eq1}), and all remaining connected components $L_i,i\ge 2$, satisfy: $|L_i|\le \frac{d}{c-1-\ln c}$.
\end{thm}


We will need the following facts.

\begin{lem}\label{le1}
For a constant $c>1$ and $d\rightarrow\infty$, the probability of survival of the Galton-Watson tree with offspring distribution $\Bin(d,\frac{c}{d})$ is asymptotic to $y$, where $y=y(c)$ is defined by (\ref{eq1}).
\end{lem}

This is very standard, see, e.g., Theorem 4.3.12 of \cite{Dur19} for a proof.

The next statement is the famed edge isoperimetric inequality for the binary cube due to Harper (1964).
\begin{lem}\label{le2}\cite{Har64}
Let $S\subset V(Q^d)$, $|S|\le 2^{d-1}$. Then the number of edges $e_{Q_d}(S,\bar{S})$ between $S$ and its complement $\bar{S}$ in $Q^d$ satisfies: $e_{Q^d}(S,\bar{S})\ge |S|(d-\log_2|S|)$. In particular, $e_{Q^d}(S,\bar{S})\ge |S|$.
\end{lem}

The next lemma bounds from above the number of trees of a given size containing a given vertex in a base graph of bounded degree.
\begin{lem}\label{le3}
Let $G$ be a graph of maximum degree $d$, and let $k>0$ be an integer. Then for every $v\in V(G)$, the number of subtrees of $G$ on $k$ vertices containing $v$ is at most $(ed)^{k-1}$.
\end{lem}

For completeness, we present its short proof here, as given by Beveridge, Frieze and McDiarmid (\cite{BFM98}, Lemma 2).

\begin{proof}
Denote by ${\cal T}(v,k)$ the set of all subtrees of $G$ of size $k$ rooted at $v$, the quantity in question is then $t(v,k):=|{\cal T}(v,k)|$. Given a tree $T\in {\cal T}(v,k)$, label $v$ with $k$, and choose a labeling $f: V(T)\setminus\{v\}\rightarrow [k-1]$ of the remaining vertices of $T$. Clearly, every $T\in {\cal T}(v,k)$ is in $(k-1)!$ many such pairs $(T,f)$. Furthermore, each such pair defines a unique spanning tree $T'$ of $K_k$, in which $(i,j)$ is an edge of $T'$ if and only if the vertices $x,y\in V(T)$ with $f(x)=i$, $f(y)=j$ are connected by an edge of $T$.

Now, we need to see in how many ways a spanning tree $T'$ of $K_k$ can be obtained under such labelings. Fix, say, a BFS order on $T'$ starting from $k$, and upon reaching vertex $\ell\in [k-1]$ for the first time, define $f^{-1}(\ell)$. For this, we need to allocate a neighbor in $G$ of the preimage of the already embedded father of $\ell$, and this can be done in at most $d$  ways. Hence, using Cayley's formula to count the number of spanning trees $T'$ of $K_k$, we obtain:
$$
\mbox{\# of pairs } (T,f)= t(v,k)\cdot (k-1)! \le d^{k-1}k^{k-2}\,,
$$
implying $t(v,k)\ge \frac{k^{k-2}}{(k-1)!}d^{k-1}$. It is easy to verify that $\frac{k^{k-2}}{(k-1)!}<e$ for every positive integer $k$, and the lemma follows.
\end{proof}


We now proceed to the proof of Theorem \ref{th1}. From now on we assume $c>1$ is a constant. We also assume that $t>0$ is a fixed integer, whose specific value will be set later in the proof.

\begin{lem}\label{le4n}
Set $p=\frac{c}{d}$, and form a random subgraph $Q^d_p$. Then \whp $Q^d_p$ has no components of size between $\frac{d}{c-1-\ln c}$ and $d^t$.
\end{lem}
\begin{proof}
Let $v\in V(Q^d)$, and denote by $C(v)$ the connected component of $Q^d_p$ containing $v$.  For $|C(v)|=k$ to happen, some tree $T$ of order $k$ in $Q^d$ containing $v$ should have all of its $k-1$ edges open in $Q^d_p$, yet all edges of $Q^d$ between $V(T)$ and its complement are to stay close in $Q^d_p$. By Lemma \ref{le2}, the number of such edges is at least $k(d-\log_2k)$. Hence, using Lemma \ref{le3} and the union bound, we derive:
\begin{eqnarray*}
\Pr\left[\frac{d}{c-1-\ln c}\le |C(v)|\le d^t\right] &\le& \sum_{k=\frac{d}{c-1-\ln c}}^{d^t} (ed)^{k-1}p^{k-1}(1-p)^{dk-k\log_2k}\\
&\le& \sum_{k=\frac{d}{c-1-\ln c}}^{d^t}
(ed)^{k-1}\left(\frac{c}{d}\right)^{k-1}e^{-\frac{c}{d}(dk-k\log_2k)}
\le \sum_{k=\frac{d}{c-1-\ln c}}^{d^t} \left( ece^{-c+\frac{c\log_2k}{d}}\right)^k\\
&=& \sum_{k=\frac{d}{c-1-\ln c}}^{d^t} \left(e^{-c+1+\ln c+o(1)}\right)^k=o(1/n)\,.
\end{eqnarray*}
Applying the union bound over all $n$ vertices of $Q^d$, we establish the lemma.
\end{proof}

\begin{lem}\label{le5}
Set $p=\frac{c}{d}$, and form a random subgraph $Q^d_p$. Let $v\in V(Q^d)$, and denote by $C(v)$ the connected component of $Q^d_p$ containing $v$. Then:
$$
\Pr[|C(v)|\ge d^t]=(1-o_d(1))y\,,
$$
where $y=y(c)$ is defined by (\ref{eq1}).
\end{lem}
\begin{proof}
First, we estimate the probability of $|C(v)|\ge d^{1/2}$. We run the BFS algorithm on $Q^d$, starting from $v$ and feeding it with independent Bernoulli$(p)$ bits, one for each queried edge of $Q^d$. For as long as $|C(v)|\ge d^{1/2}$, every vertex $u\in V(Q^d)$, queried for neighbors outside of the current component $C(v)$, has at least $d-d^{1/2}$ potential neighbors to query. Hence the exploration process can be coupled with the Galton-Watson tree rooted at $v$ with offspring distribution $\Bin(d-d^{1/2},p)$. Since $(d-d^{1/2})p=c-o_d(1)$, by Lemma \ref{le1} the component $C(v)$ grows to $d^{1/2}$ with probability asymptotic to $y$.

Now we estimate $\Pr[d^{1/2}\le |C(v)|\le d^t]$. The argument here is nearly identical to that of Lemma \ref{le4n}, and hence we allow ourselves to be brief. We have:
\begin{eqnarray*}
\Pr[d^{1/2}\le |C(v)|\le d^t] &\le& \sum_{k=d^{1/2}}^{d^t} (ed)^{k-1}p^{k-1}(1-p)^{dk-k\log_2k}
\le \sum_{k=d^{1/2}}^{d^t}
(ed)^{k-1}\left(\frac{c}{d}\right)^{k-1}e^{-\frac{c}{d}(dk-k\log_2k)}\\
&\le& \sum_{k=d^{1/2}}^{d^t} \left( ece^{-c+\frac{c\log_2k}{d}}\right)^k=o(1)\,,
\end{eqnarray*}
since $ece^{-c}<1$ for $c>1$. We thus conclude:
$$
\Pr[|C(v)|\ge d^t]=(1-o_d(1))y\,,
$$
as desired.
\end{proof}

\begin{lem}\label{le6}
Set $p=\frac{c}{d}$, and form a random subgraph $Q^d_p$. Let $W=\{v: |C(v)|\ge d^t\}$. Then {\bf whp},
$$
|W|=(1+o(1))yn\,.
$$
\end{lem}

\begin{proof}
By the previous lemma, we have $\E[|W|]=(1+o(1))yn$. In order to argue about concentration, apply edge exposure martingale to $Q^d_p$ (see, e.g., Chapter 7 of \cite{AS16}). Adding or deleting an edge can change the value of $|W|$ by at most $2d^t$. Hence, by the standard edge exposure martingale inequality:
$$
\Pr\left[||W|-\E[|W|]|\ge n^{2/3}\right]\le 2\exp\left\{-\frac{n^{4/3}}{2\cdot \frac{nd}{2}\cdot(2d^t)^2}\right\}= o(1).
$$
The lemma follows.
\end{proof}


We now argue that typically in $Q^d_p$ ``largish" components are well spread.
\begin{lem}\label{le7n}
Set $p=\frac{c}{d}$, and form a random subgraph $Q^d_p$. Let $W=\{v: |C(v)|\ge d^t\}$. Then \whp every vertex $v\in V(Q^d)$ is at distance at most two (in $Q^d$) from $W$.
\end{lem}
\begin{proof}
It is enough to prove that the all-zero vertex $\bar{0}$ is at distance at most two from $W$ with probability $1-o(1/n)$.
Set $\eps=\frac{c-1}{c}$.
Let $I$ be the set of first $\left\lfloor\frac{\eps d}{2}\right\rfloor$ coordinates, and for $i\ne j\in I$, look at the subcube
$$
H_{ij}=\{v\in Q^d: v_i=v_j=1, v_k=0 \mbox{ for all } k\in I\setminus\{i,j\}\}\,.
$$
Let $u_{ij}$ be the vertex with 1 in coordinates $i$ and $j$ and 0 in all other coordinates. Clearly, $u_{ij}$ is at distance two from $\bar{0}$, and $u_{ij}$ belongs to $H_{ij}$.

The subcubes $H_{ij}$ are vertex disjoint and have dimension $d'=d-\left\lfloor\frac{\eps d}{2}\right\rfloor \ge \left(1-\frac{\eps}{2}\right)d$. Performing percolation on $H_{ij}$ with $p=\frac{c}{d}$ and observing that $d'p\ge \left(1-\frac{\eps}{2}\right)dp=\left(1-\frac{\eps}{2}\right)c=\frac{c+1}{2}$, we derive from Lemma~\ref{le5} that $\Pr[|C(u_{ij})|\ge d^t]=(1+o(1))y(d'p)= (1+o(1))y\left(\frac{c+1}{2}\right)\ge \delta$ for some $\delta=\delta(c)>0$. These events are independent for distinct pairs $\{i,j\}$. Hence the probability that $\bar{0}$ is not at distance at most two (in $Q^d$) from $W$ is at most $\Pr\left[\Bin\left(\binom{\left\lfloor\frac{\eps d}{2}\right\rfloor}{2},\delta\right)=0\right]=e^{-\Theta(d^2)}=o(1/n)$.
\end{proof}


Now it is time for the final attack. Consider $Q^d_p$, $p=\frac{c}{d}$. We define probabilities $p_1,p_2$ as follows:
$$
p_2=\frac{1}{d^5},\quad 1-p=(1-p_1)(1-p_2)\,,
$$
we have $p_1=\frac{c-o(1)}{d}$. Let $G_1\sim Q^d_{p_1}, G_2\sim Q^d_{p_2}$ be independent random subgraphs of $Q_p$, their union is distributed exactly as $Q^d_p$. Set
$$
t=31\,,
$$
and define
$$
W_1 =\{v\in V(Q^d): |C_{G_1}(v)|\ge d^t\}\,.
$$
By Lemma \ref{le6}, \whp $|W_1|=(1+o(1))yn$. By Lemma \ref{le4n}, \whp components of $G_1$ outside $W_1$ are $O(d)$ in sizes.

We now expose the edges of $G_2$.

\begin{lem}\label{le11}
{\bf Whp} all components of $G_1[W_1]$ merge into a single component in $Q^d_p$.
\end{lem}

\begin{proof}
The proof follows the pioneering idea of \cite{AKS82}. Recall that the edges of $Q^d$ appear as edges of $G_2$ independently and with probability $p_2=d^{-5}$. By Lemma \ref{le7n}, we can assume that every $v\in V(Q^d)$ is at distance at most two from $W_1$. If the statement does not hold, then we can partition the components of $G_1[W_1]$ into two families $\cA$ and $\BB$ such that there are no paths in $G_2$ between sets $A:=\cup_{C\in\cA} V(C)$ and $B:=\cup_{C\in\BB} V(C)$. Let $s\le n$ be the total number of components in $G_1[W_1]$, and let $\ell\le s/2$ be the number of components in a family with fewer components. Then $|A|,|B|\ge \ell\cdot d^{31}$.

Now, since every vertex of $Q^d$ is at distance at most two from $A\cup B$, we can partition $V(Q^d)$ into sets $A',B'$, where $A\subseteq A'$ and $A'$ contains all vertices of $V(Q^d)\setminus B$ at distance at most two from $A$, and every vertex in $B':=V(Q^d)\setminus A'$ is at distance at most two from $B$. By Harper's edge isoperimetric inequality (Lemma \ref{le2}), $Q^d$ has at least $\min\{|A'|,|B'|\}\ge \ell d^{31}$ edges between $A'$ and $B'$. Since every vertex in $A'$ is at distance at most two from $A$ and every vertex in $B'$ is at distance at most two from $B$, we can extend each edge crossing between $A'$ and $B'$ to a path of length at most five between $A$ and $B$. As every edge of $Q^d$ is contained in less than $5d^4$ paths of length at most five, by applying a simple greedy argument we obtain a family of $\frac{\ell d^{31}}{5\cdot 5d^4+1}\ge \frac{\ell d^{27}}{30}$ edge-disjoint paths of length at most five between $A$ and $B$. The probability that none of these paths is in $G_2$ is at most $(1-p_2^5)^{\frac{\ell d^{27}}{30}}\le e^{-\frac{\ell d^2}{30}}$. Also, $s\le n$, and hence the number of ways to partition the components of $G_1[W_1]$ into $\cA,\BB$ with $\ell$ components in one of the families is at most $\binom{s}{\ell}\le\binom{n}{\ell}$. Thus, by the union bound the probability that the lemma's statement fails is most
$$
\sum_{\ell=1}^{s/2}\binom{s}{\ell}e^{-\frac{\ell d^2}{30}}
\le \sum_{\ell=1}^{s/2}\left(\frac{en}{\ell}\cdot e^{-\frac{d^2}{30}}\right)^{\ell}
= o(1)\,.
$$
\end{proof}

\noindent{\bf Proof of Theorem \ref{th1}.} By Lemma \ref{le6}, \whp $|W_1|=(1+o(1))yn$. By Lemma \ref{le11}, \whp all components in $W_1$ merge into a single component $L_1$, whose size is then at least $(1+o(1))yn$. Let us now look at the situation with components outsize of $W_1$. Define an auxiliary graph $\Gamma$ whose vertices are the components of $G_1$ outside of $W_1$, where two components are connected by an edge in $\Gamma$ if $G_2$ contains at least one edge between them. By Lemma \ref{le4n}, \whp all components of $G_1$ outside of $W_1$ are of size $O(d)$, and thus, using a very crude estimate, there are \whp  $O(d^2)$ edges of $Q^d$ between every pair of connected components. Hence every pair of components is connected by an edge of $\Gamma$ independently and with probability at most  $1-(1-p_2)^{d^2}\le p_2d^2=O(d^{-3})$. By the same reasoning, the maximum degree of the underlying graph of $\Gamma$ is \whp $O(d^2)$. Hence, applying Theorem \ref{th0} we derive that \whp all components of $\Gamma$ are $O(d)$. Recalling that \whp every component of $G_1$ outside of $W_1$ has $O(d)$ vertices, this yields that \whp all components of $Q^d_p$ outside of $W_1$ are $O(d^2)$ in sizes. Invoking Lemma \ref{le4n} once again, we conclude that in fact \whp all such components $L_i$ satisfy: $|L_i|\le \frac{d}{c-1-\ln c}$. Finally, applying Lemma \ref{le6}, this time to $Q^d_p$, we can give a likely upper bound on $|L_1|$: $|L_1|\le (1+o(1))yn$. The theorem follows.
\hfill $\Box$

\medskip

\noindent{\bf Remark.} The same proof, with minimal changes, works also for (some part of) the slightly supercritical regime $p=\frac{1+\eps}{d}$, for $d^{-a}\le \eps=\eps(d)=o_d(1)$, for some constant $a>0$.

\medskip

\noindent{\bf Acknowledgment.} The author wishes to thank his frequent coauthors of recent papers on percolation on finite graphs Sahar Diskin, Joshua Erde and Mihyun Kang for the fruitful and enjoyable cooperation, which spurred this note.


\begin{thebibliography}{9}
\bibitem{AKS82}
M. Ajtai, J. Koml\'os and E. Szemer\'edi,
Largest random component of a $k$-cube,
Combinatorica 2(1982), 1--7.

\bibitem{AS16}
N. Alon and J. H. Spencer,
{\em  The probabilistic method},
4th edition, Wiley, Hoboken, 2016.

\bibitem{BFM98}
A. Beveridge, A. Frieze, and C. McDiarmid,
Random minimum length spanning trees in regular graphs,
Combinatorica 18 (1998), 311--333.

\bibitem{BKL92}
B. Bollob\'as, Y. Kohayakawa and T. \L uczak,
The evolution of random subgraphs of the cube,
Random Structures \and Algorithms 3 (1992), 55--90.


\bibitem{Dur19}
R. Durrett,
{\em  Probability: theory and examples},
Cambridge University Press, Cambridge, 2019.

\bibitem{Har64}
L. H. Harper,
Optimal assigments of numbers to vertices,
SIAM J. Appl. Math. 12 (1964), 131--135.

\end{thebibliography}
\end{document}